\documentclass[a4paper]{amsart}

\usepackage[colorlinks=true]{hyperref}
\usepackage{amsmath}
\usepackage{amsfonts}
\usepackage{amssymb}

\newtheorem{theorem}{Theorem}[section]

\newtheorem{corollary}[theorem]{Corollary}
\newtheorem{lemma}[theorem]{Lemma}

\theoremstyle{definition}

\numberwithin{equation}{section}

\begin{document}

\title{Curvature properties of 3-quasi-Sasakian manifolds}

\author[B. Cappelletti Montano]{Beniamino Cappelletti Montano}
 \address{Dipartimento di Matematica e Informatica, Universit\`a degli Studi di
 Cagliari, Via Ospedale 72, 09124 Cagliari, Italy}
 \email{b.cappellettimontano@gmail.com}

\author[A. De Nicola]{Antonio De Nicola}
 \address{CMUC, Department of Mathematics, University of Coimbra, 3001-454 Coimbra, Portugal}
 \email{antondenicola@gmail.com}

\author[I. Yudin]{Ivan Yudin}
 \address{CMUC, Department of Mathematics, University of Coimbra, 3001-454 Coimbra, Portugal}
 \email{yudin@mat.uc.pt}



\thanks{Research partially supported by CMUC and FCT (Portugal), through
European program COMPETE/FEDER, grants PTDC/MAT/099880/2008 and PEst-C/MAT/UI0324/2011,
MTM2009-13383 (A.D.N.), and SFRH/BPD/31788/2006 (I.Y.).}

\begin{abstract}
We find some curvature properties of $3$-quasi-Sasakian manifolds
which are similar to some well-known identities holding in the
Sasakian case. As an application, we prove that any
$3$-quasi-Sasakian manifold of constant horizontal sectional
curvature is necessarily either $3$-$\alpha$-Sasakian or
$3$-cosymplectic.
\end{abstract}

\maketitle

\section{Introduction}
An important topic in contact Riemannian geometry is the study of curvature properties of almost contact metric manifolds (see \cite{blairbook} for
details). In some cases it is in fact possible to characterize a manifold in terms of its curvature tensor field. The typical example is given by Sasakian
manifolds, which are characterized by the well-known condition
\begin{equation}\label{sasaki0}
R_{X Y}\xi = \eta(Y)X - \eta(X)Y.
\end{equation}
A key role in this area is played by the interaction between the
curvature and the structure tensors $(\phi,\xi,\eta)$ of an almost
contact metric manifold. For instance, in any Sasakian manifold one
has
\begin{align}
R(X,Y,Z,W)&=R(X,Y,\phi Z,\phi W) - g(X, Z)g(Y, W)+g(X, W)g(Y, Z) \nonumber \\
&\quad +g(X, \phi Z)g(Y, \phi W)-g(X, \phi W)g(Y, \phi Z)\label{sasaki1}
\end{align}
and in any cosymplectic manifold
\begin{equation}\label{cosymplectic1}
R(X,Y,Z,W)=R(X,Y,\phi Z,\phi W)
\end{equation}
for any vector fields $X,Y,Z,W$. The relations \eqref{sasaki1} and
\eqref{cosymplectic1} turn out to be useful for studying the
$\phi$-sectional curvature and the Ricci tensor and deriving other
properties on the geometry of the manifold. A generalization of
\eqref{sasaki1} and \eqref{cosymplectic1} was proposed by Janssens
and Vanecke in \cite{janssens-vanecke81}. They defined a
\emph{$C(\alpha)$-manifold} as a normal almost contact metric
manifold whose curvature tensor satisfies the condition
\begin{align*}
R(X,Y,Z,W)&=R(X,Y,\phi Z,\phi W) + \alpha\bigl(- g(X, Z)g(Y, W)+g(X, W)g(Y, Z) \nonumber \\
&\quad +g(X, \phi Z)g(Y, \phi W)-g(X, \phi W)g(Y, \phi Z)\bigr),
\end{align*}
for some $\alpha\in{\mathbb{R}}$. $C(\alpha)$-manifolds include
Sasakian, cosymplectic and Kenmotsu manifolds. Another
generalization, due to Blair, is given by the notion of
\emph{quasi-Sasakian structure} (\cite{blair0}). By definition, a
quasi-Sasakian manifold is a normal almost contact metric manifold
whose fundamental $2$-form $\Phi:=g(\cdot,\phi\cdot)$ is closed.
This class includes Sasakian and cosymplectic manifolds and can be
viewed as an odd-dimensional counterpart of K\"{a}hler structures.
Although quasi-Sasakian manifolds were studied by several different
authors and are considered a well-established topic in contact
Riemannian geometry, only little about their curvature properties is
known. With this regard we mention the attempts of Olszak
(\cite{olszak82}) and Rustanov (\cite{rustanov94}). On the other
hand, if a quasi-Sasakian manifold is endowed with two additional
quasi-Sasakian structures defining a $3$-quasi-Sasakian manifold
then, as shown in \cite{agag08} and \cite{ijm09}, the
quaternionic-like relations force the three structures to satisfy
more restrictive geometric conditions.

Motivated by these considerations, in this paper we study the
curvature properties of $3$-quasi-Sasakian manifolds. We are able to
find conditions similar to \eqref{sasaki0}, \eqref{sasaki1}, and
\eqref{cosymplectic1} for the $3$-quasi-Sasakian case. Moreover, we
present one application of these properties by proving a formula
relating the three $\phi$-sectional curvatures of a
$3$-quasi-Sasakian manifold. We then obtain that a $3$-quasi-Sasakian
manifold has constant horizontal sectional curvature if and only if
it is either 3-$c$-Sasakian or 3-cosymplectic. In the first case it
is a space of constant curvature $c^2/4$ and in the latter case it is
flat. The last result extends to the quasi-Sasakian setting a famous
theorem of Konishi (\cite{konishi76}).

\section{Preliminaries}
A quasi-Sasakian manifold $(M,\phi,\xi,\eta,g)$ of dimension $2n+1$
is said to be of rank $2p$ (for some $p\leq n$) if
$\left(d\eta\right)^p\neq 0$ and $\eta\wedge\left(d\eta\right)^p=0$
on $M$, and to be of rank $2p+1$ if
$\eta\wedge\left(d\eta\right)^p\neq 0$ and
$\left(d\eta\right)^{p+1}=0$  on $M$ (cf. \cite{blair0, tanno}).  It
was proven in \cite{blair0} that there are no quasi-Sasakian
manifolds of (constant) even rank.  Particular subclasses of
quasi-Sasakian manifolds are  \emph{$c$-Sasakian manifolds} (usually called $\alpha$-Sasakian), which have rank $2n+1$,  and \emph{cosymplectic manifolds} (rank $1$) according to
satisfy, in addition, $d\eta=c\Phi$ $(c \neq 0)$ and $d\eta=0$,
respectively. For $c=2$ we obtain the well-known Sasakian manifolds.

If on the same manifold $M$ there are given three distinct almost
contact structures  $\left(\phi_1,\xi_1,\eta_1\right)$,
$\left(\phi_2,\xi_2,\eta_2\right)$,
$\left(\phi_3,\xi_3,\eta_3\right)$ satisfying the following
relations, for any even permutation
$\left(\alpha,\beta,\gamma\right)$ of $\left\{1,2,3\right\}$,
\begin{equation} \label{3-structure}
\begin{split}
\phi_\gamma=\phi_{\alpha}\phi_{\beta}-\eta_{\beta}\otimes\xi_{\alpha}=-\phi_{\beta}\phi_{\alpha}+\eta_{\alpha}\otimes\xi_{\beta},\quad\\
\xi_{\gamma}=\phi_{\alpha}\xi_{\beta}=-\phi_{\beta}\xi_{\alpha}, \ \ \eta_{\gamma}= \eta_{\alpha}\circ\phi_{\beta}=-\eta_{\beta}\circ\phi_{\alpha},
\end{split}
\end{equation}
we say that $(\phi_\alpha,\xi_\alpha,\eta_\alpha)$,
$\alpha\in\left\{1,2,3\right\}$, is an almost contact $3$-structure.
Then the dimension of $M$ is necessarily of the form $4n+3$. This
notion was introduced independently by Kuo (\cite{kuo}) and Udriste
(\cite{udriste}). An almost $3$-contact manifold $M$ is said to be
\emph{hyper-normal} if each almost contact structure
$\left(\phi_\alpha,\xi_\alpha,\eta_\alpha\right)$ is normal.

In \cite{kuo} Kuo proved that given an almost contact $3$-structure
$\left(\phi_\alpha,\xi_\alpha,\eta_\alpha\right)$, there exists a
Riemannian metric $g$ compatible with each of the three almost
contact structure and hence we can speak of \emph{almost contact
metric $3$-structure}. It is well known that in any almost
$3$-contact metric manifold the Reeb vector fields
$\xi_1,\xi_2,\xi_3$ are orthonormal with respect to the compatible
metric $g$. Moreover, by putting
${\mathcal{H}}=\bigcap_{\alpha=1}^{3}\ker\left(\eta_\alpha\right)$
we obtain a codimension $3$ distribution on $M$ and the tangent
bundle splits as the orthogonal sum
$TM={\mathcal{H}}\oplus{\mathcal{V}}$, where ${\mathcal
V}=\left\langle\xi_1,\xi_2,\xi_3\right\rangle$. The distributions
$\mathcal H$ and $\mathcal V$ are called, respectively,
\emph{horizontal} and \emph{Reeb distribution}.

A \emph{$3$-quasi-Sasakian structure} is  an almost contact metric
$3$-structure such that each structure
$(\phi_\alpha,\xi_\alpha,\eta_\alpha,g)$ is quasi-Sasakian.
Remarkable subclasses are  $3$-Sasakian and $3$-cosymplectic
manifolds. Another subclass of $3$-quasi-Sasakian structures is
given by almost contact metric $3$-structures
$(\phi_\alpha,\xi_\alpha,\eta_\alpha,g)$ such that each structure
$(\phi_\alpha,\xi_\alpha,\eta_\alpha,g)$ is $c_{\alpha}$-Sasakian.
It is proven in \cite{kashiwada2} that the non-zero constants $c_1$,
$c_2$, $c_3$ must coincide. Therefore we speak of
\emph{$3$-$c$-Sasakian manifolds}. Many results on
$3$-quasi-Sasakian manifolds were obtained in \cite{agag08} and
\cite{ijm09}. We collect some of them in the following theorem.
\begin{theorem}[\cite{agag08, ijm09}]\label{principale}
Let $(M,\phi_\alpha,\xi_\alpha,\eta_\alpha,g)$ be a $3$-quasi-Sasakian manifold of
dimension $4n+3$. Then, for any even permutation
$(\alpha,\beta,\gamma)$ of $\left\{1,2,3\right\}$, the Reeb vector
fields satisfy
\begin{equation}\label{lie}
 [\xi_\alpha,\xi_\beta]=c\xi_\gamma,
\end{equation}
for some $c\in\mathbb{R}$. Moreover, the $1$-forms $\eta_1$,
$\eta_2$, $\eta_3$ have the same rank,  called the \emph{rank} of
the $3$-quasi-Sasakian manifold $M$. The rank of $M$ is $1$ if and
only if $M$ is $3$-cosymplectic and it is an integer of the form
$4l+3$, for some $l\leq n$, in the other cases. Furthermore, any
$3$-quasi-Sasakian manifold of rank $4n+3$ is necessarily
$3$-$c$-Sasakian.
\end{theorem}

We point out that the constant $c$ in \eqref{lie} is zero if and
only if the manifold is $3$-cosymplectic. Moreover, for any
$3$-quasi-Sasakian manifold of rank $4l+3$ one can consider the
distribution
\begin{equation*}
{\mathcal E}^{4m}:=\left\{X\in{\mathcal H} | i_{X}\eta_{\alpha}=0, i_{X}d\eta_{\alpha}=0 \ \hbox{for any $\alpha=1,2,3$} \right\} \ \  \ (l+m=n)
\end{equation*}
and its orthogonal complement ${\mathcal E}^{4l+3}:=({\mathcal E}^{4m})^{\perp}$. We will also consider the distribution ${\mathcal E}^{4l}$ which is the
orthogonal complement of $\mathcal V$ in ${\mathcal E}^{4l+3}$. A remarkable property of $3$-quasi-Sasakian manifolds, which in general does not hold for a
single quasi-Sasakian structure, is that both ${\mathcal E}^{4l+3}$ and ${\mathcal E}^{4m}$ are integrable and define Riemannian foliations with totally
geodesic leaves. In particular it follows that $\nabla{\mathcal E}^{4l+3}\subset{\mathcal E}^{4l+3}$ and $\nabla{\mathcal E}^{4m}\subset{\mathcal E}^{4m}$.

\smallskip

{\small All manifolds considered in the paper are assumed to be connected. The Spivak's conventions for the differential, the wedge
product and the interior product are adopted.}

\section{Main results}
We recall that in any 3-quasi-Sasakian manifold  of rank $4l+3$  for
each $\alpha\in\left\{1,2,3\right\}$ one defines two  tensors
$\psi_\alpha$ and $\theta_\alpha$ by
\begin{equation*}
\psi_\alpha:=\left\{
               \begin{array}{ll}
                 \phi_\alpha, & \hbox{on ${\mathcal E}^{4l+3}$} \\
                 0, & \hbox{on ${\mathcal E}^{4m}$}
               \end{array}
             \right. \ \ \
\theta_\alpha:=\left\{
               \begin{array}{ll}
                 0, & \hbox{on ${\mathcal E}^{4l+3}$} \\
                 \phi_\alpha, & \hbox{on ${\mathcal E}^{4m}$.}
               \end{array}
             \right.
\end{equation*}
Moreover we define $\Psi_\alpha(X,Y):=g(X,\psi_\alpha Y)$ and $\Theta_\alpha(X,Y):=g(X,\theta_\alpha Y)$ for all $X,Y\in\Gamma(TM)$. The tensors
$\psi_\alpha$ and $\Psi_\alpha$ satisfy
\begin{equation}\label{differential1}
d\eta_\alpha = c \Psi_\alpha,   \qquad  \nabla\xi_\alpha=-\frac{c}{2}\psi_\alpha
\end{equation}
 (cf. \cite[(4.8)]{ijm09} and \cite[Theorem 4.3]{ijm09}). Since  $\phi_\alpha=\psi_\alpha + \theta_\alpha$ one has that $\Phi_\alpha=\Psi_\alpha + \Theta_\alpha$. Consequently, due
to \eqref{differential1}, $\Psi_\alpha$ and $\Theta_\alpha$ are closed $2$-forms. We start with a few  lemmas. The first is immediate.

\begin{lemma}
In any $3$-quasi-Sasakian manifold  of rank $4l+3$  one has,
\begin{gather}\label{simmetry}
g(\psi_\alpha^2 X,Y)=g(X,\psi_\alpha^2Y),\\
\psi_{\alpha}^{3}=-\psi_{\alpha}\label{psicube},\\
\nabla\eta_\alpha=\frac{c}{2}\Psi_\alpha.
\end{gather}
\end{lemma}

\begin{lemma}
In any 3-quasi-Sasakian manifold  of rank $4l+3$  one has
\begin{equation}\label{nablapsi}
(\nabla_{X}\psi_\alpha)Y=\frac{c}{2}\left(\eta_\alpha(Y)\psi_\alpha^2 X - g(\psi_\alpha^2 X,Y)\xi_\alpha\right).
\end{equation}
\end{lemma}
\begin{proof}Let $X\in\Gamma(TM)$.
According to the orthogonal decomposition $TM={\mathcal E}^{4l+3}\oplus{\mathcal E}^{4m}$  we may distinguish the following two cases. (i) Assume
$Y\in\Gamma({\mathcal E}^{4l+3})$. Then, since ${\nabla \mathcal E}^{4l+3} \subset {\mathcal E}^{4l+3}$, we have $(\nabla_{X}\psi_\alpha)Y =
\nabla_{X}(\psi_{\alpha}Y) - \psi_\alpha \nabla_{X}Y = \nabla_{X}(\phi_{\alpha}Y) - \phi_\alpha \nabla_{X}Y = (\nabla_{X}\phi_\alpha)Y$. The assertion then
follows from \cite[(4.9)]{ijm09}. (ii) If $Y\in\Gamma({\mathcal E}^{4m})$, then, as ${\nabla \mathcal E}^{4m} \subset {\mathcal E}^{4m}$, one has
$(\nabla_{X}\psi_\alpha)Y=\nabla_{X}(\psi_\alpha Y) - \psi_\alpha \nabla_{X}Y =0$. On the other hand, by using
\eqref{simmetry} and $Y \in \Gamma\left( \mathcal{E}^{4m} \right) \subset
\ker\left( \eta_\alpha \right)\cap \ker\left( \psi_\alpha \right)$, one has
\begin{align*}
\frac{c}{2}\left(\eta_{\alpha}(Y)\psi_\alpha^2 X - g(\psi_\alpha^2
X,Y)\xi_\alpha\right) = - \frac{c}2 g\left( X, \psi_\alpha^2 Y \right)\xi_{\alpha} =0.
\end{align*}
\end{proof}

By using \eqref{nablapsi} and \eqref{psicube} we get straightforwardly the following formula for $\nabla\psi_{\alpha}^2$.
\begin{lemma}
In any 3-quasi-Sasakian manifold   of rank $4l+3$  one has
\begin{equation}\label{nablaphisquare}
(\nabla_{X}\psi_{\alpha}^{2})Y=\frac{c}{2}\left(\Psi_{\alpha}(X,Y)\xi_\alpha-\eta_{\alpha}(Y)\psi_{\alpha}X\right).
\end{equation}
\end{lemma}
\begin{theorem}\label{curvature5}
In any 3-quasi-Sasakian manifold the following formula holds
\begin{equation*}
R_{XY}\xi_\alpha=\frac{c^2}{4}\left(\eta_{\alpha}(X)\psi_{\alpha}^{2}Y - \eta_{\alpha}(Y)\psi_{\alpha}^{2}X\right).
\end{equation*}
\end{theorem}
\begin{proof}
If the manifold is 3-cosymplectic, i.e. $c=0$, the claim follows
easily from the property that each $\xi_\alpha$ is parallel. Thus we
can assume that $M$ has rank $4l+3$. By using \eqref{differential1}, \eqref{nablapsi}, and \eqref{simmetry}, we have
\begin{align*}
R_{XY}\xi_\alpha &= \frac{c}{2}\left(\nabla_{Y}(\psi_{\alpha}X) -\nabla_{X}(\psi_{\alpha}Y) + \psi_{\alpha}[X,Y]\right)\\
&=\frac{c}{2}\left((\nabla_{Y}\psi_{\alpha})X - (\nabla_{X}\psi_{\alpha})Y\right)\\
&=\frac{c^2}{4}\left(\eta_{\alpha}(X)\psi_{\alpha}^{2}Y - g(\psi_{\alpha}^{2}Y,X)\xi_{\alpha} - \eta_{\alpha}(Y)\psi_{\alpha}^{2}X + g(\psi_{\alpha}^{2}X,Y)\xi_{\alpha}\right)\\
&=\frac{c^{2}}{4}\left(\eta_{\alpha}(X)\psi_{\alpha}^{2}Y-\eta_{\alpha}(Y)\psi_{\alpha}^{2}X\right).
\end{align*}
\end{proof}
\begin{theorem}\label{curvature1}
Let $M$ be a 3-quasi-Sasakian manifold  of rank $4l+3$. Then,
\begin{align*}
R_{XY}\phi_\alpha Z - \phi_\alpha R_{XY}Z & = \frac{c^2}{4}\bigl(\left(\Psi_\alpha(Y,\psi_\alpha Z)
-\eta_\alpha(Y)\eta_\alpha(Z)\right)\psi_\alpha X - \left(\Psi_\alpha(X,\psi_\alpha Z)\right.\\
&\quad \left. -\eta_\alpha(X)\eta_\alpha(Z)\right)\psi_\alpha Y -\Psi_\alpha(Y,Z)\psi_\alpha^2 X + \Psi_\alpha(X,Z)\psi_\alpha^2
Y\\
&\quad +\left(\eta_\alpha(X)\Psi_\alpha(Y,Z)-\eta_\alpha(Y)\Psi_\alpha(X,Z)\right)\xi_\alpha\bigr).
\end{align*}
\end{theorem}
\begin{proof}
The claim follows from a long computation using \eqref{nablapsi},
\eqref{nablaphisquare} and \eqref{psicube}.
\end{proof}

\begin{corollary}\label{curvature3}
In any 3-quasi-Sasakian manifold  of rank $4l+3$  one has %the following relation holds, for any $X,Y,Z,W\in\Gamma(TM)$,
\begin{equation*}
g(R_{XY}\phi_\alpha Z,W)+g(R_{XY}Z,\phi_\alpha W)=-P_{\alpha}(X,Y,Z,W),
\end{equation*}
where $P_\alpha$ is the tensor defined by
\begin{align*}
P_{\alpha}(X,Y,Z,W)&=\frac{c^2}{4}\bigl(\Psi_\alpha(Y,Z)\Psi_\alpha(X,\psi_\alpha W)-\Psi_\alpha(X,Z)\Psi_\alpha(Y,\psi_\alpha W)\\
&\quad+\Psi_\alpha(Y,\psi_\alpha Z)\Psi_\alpha(X,W)-\Psi_\alpha(X,\psi_\alpha Z)\Psi_\alpha(Y,W)\\
&\quad-\eta_\alpha(X)\eta_\alpha(W)\Psi_\alpha(Y,Z)-\eta_\alpha(Y)\eta_\alpha(Z)\Psi_\alpha(X,W)\\
&\quad+\eta_\alpha(Y)\eta_\alpha(W)\Psi_\alpha(X,Z)+\eta_\alpha(X)\eta_\alpha(Z)\Psi_\alpha(Y,W)\bigr).
\end{align*}
\end{corollary}

\begin{corollary}\label{curvature4}
In any 3-quasi-Sasakian manifold of rank $4l+3$ one has
\begin{align*}
g(R_{\phi_\alpha X \phi_\alpha Y}\phi_\alpha Z,\phi_\alpha W)&=\frac{c^2}{4}\bigl(g(R_{X Y}Z,W) + \Psi_\alpha(Z,X)\Psi_\alpha(W,\psi_\alpha\phi_\alpha Y)\\
&\quad +\Psi_\alpha(Z,\psi_\alpha X)\Psi_\alpha(W,\phi_\alpha Y) \\
&\quad + \Psi_\alpha(\phi_\alpha X,Z)\Psi_\alpha(\phi_\alpha Y,\psi_\alpha\phi_\alpha
W)\\
&\quad +\Psi_\alpha(\phi_\alpha X,\psi_\alpha Z)\Psi_\alpha(\phi_\alpha Y,\phi_\alpha W)\bigr),
\end{align*}
for any $X,Y,Z,W\in\Gamma({\mathcal H})$.
\end{corollary}
\begin{proof}
By using Corollary \ref{curvature3} twice, one obtains
\begin{align*}
g(R_{\phi_\alpha X \phi_\alpha Y}\phi_\alpha Z,\phi_\alpha W)=g(R_{X Y}Z,W)-P_\alpha( Z,W,X,\phi_\alpha Y)
- P_\alpha(\phi_\alpha X,\phi_\alpha Y,Z,\phi_\alpha W).
\end{align*}
Next, by using \eqref{simmetry} and the property that $\phi_\alpha$
and $\psi_\alpha$ commute, we get that
\begin{align*}
P_\alpha(Z,W,X,\phi_\alpha Y) + P_\alpha(\phi_\alpha X,\phi_\alpha Y,Z,\phi_\alpha W) &=-\frac{c^2}{4}\bigl(
\Psi_\alpha(Z,X)\Psi_\alpha(W,\psi_\alpha\phi_\alpha Y)\\
&\quad +\Psi_\alpha(Z,\psi_\alpha X)\Psi_\alpha(W,\psi_\alpha Y)\\
&\quad +\Psi_\alpha(\phi_\alpha X,Z)\Psi_\alpha(\phi_\alpha Y,\psi_\alpha\phi_\alpha W)\\
&\quad +\Psi_\alpha(\phi_\alpha X,\psi_\alpha Z)\Psi_\alpha(\phi_\alpha Y,\phi_\alpha W)\bigr).
\end{align*}
Thus the assertion follows.
\end{proof}

We recall that on an almost contact metric manifold
$(M,\phi,\xi,\eta,g)$ one defines a \emph{$\phi$-section} as the
$2$-plane spanned by $X$ and $\phi X$, where $X$ is a unit vector
field orthogonal to $\xi$. Then the sectional curvature
$H(X):=K(X,\phi X)=g(R_{X \phi X}\phi X,X)$ is called
\emph{$\phi$-sectional curvature}. In a 3-quasi-Sasakian manifold
$M$, we denote by $H_\alpha$ the $\phi_\alpha$-sectional curvature.
\begin{theorem}\label{curvature2}
For any $X\in\Gamma({\mathcal H})$ the $\phi_\alpha$-sectional curvatures of a 3-quasi-Sasakian manifold  of rank $4l+3$  satisfy the following relation
\begin{equation}\label{sectional}
H_{1}(X) + H_{2}(X) + H_{3}(X) = \frac{3c^2}{4}g(X_{{\mathcal
E}^{4l}},X_{{\mathcal E}^{4l}})^2,
\end{equation}\label{sectional1}
where $X_{{\mathcal E}^{4l}}$ denotes the projection of $X$  onto the distribution ${\mathcal E}^{4l}$. In particular,
\begin{equation}\label{sectional3qS}
H_{1}(X) + H_{2}(X) + H_{3}(X) =  \left\{
                                    \begin{array}{ll}
                                      \frac{3c^2}{4}, & \hbox{for any $X\in\Gamma({\mathcal E}^{4l})$;} \\
                                      0, & \hbox{for any $X\in\Gamma({\mathcal E}^{4m})$.}
                                    \end{array}
                                  \right.
\end{equation}
\end{theorem}
\begin{proof}
From Corollary \ref{curvature3} it follows that, for any
$X,Y,Z,W\in\Gamma(\mathcal{H})$,
\begin{align}
g(R_{XY}\phi_\alpha Z,\phi_\alpha W)&=g(R_{XY}Z, W) + \frac{c^2}{4}\bigl(\Psi_\alpha(Y,\psi_\alpha Z)g(\psi_\alpha X,\phi_\alpha W)  \nonumber \\
&\quad-\Psi_\alpha(X,\psi_\alpha Z)g(\psi_\alpha Y,\phi_\alpha W)-\Psi_\alpha(Y,Z)g(\psi_\alpha^2 X,\phi_\alpha W) \label{intermediate1}\\
 &\quad +\Psi_\alpha(X,Z)g(\psi_\alpha^2
Y,\phi_\alpha W)\bigr). \nonumber
\end{align}
In \eqref{intermediate1} we put  $\alpha=1$, $Z=X$ and $Y=W=\phi_3 X$, getting
\begin{align*}
-g(R_{X \phi_3 X}\phi_1 X,\phi_2 X) & = g(R_{X \phi_3 X}X,\phi_3 X) + \frac{c^2}{4}\bigl(-g(\phi_3 X,\psi_1^2 X)g(\psi_1 X,\phi_2 X)\\
 &\quad + g(X,\psi_1^2
X)g(\psi_1 \phi_3 X,\phi_2 X) + g(\phi_3 X, \psi_1 X) g(\psi_1^2 X,\phi_2 X)\\
 &\quad  - g(X,\psi_1 X)g(\psi_1^2 \phi_3 X,\phi_2 X)\bigr).
\end{align*}
By using the definition of the operators $\psi_\alpha$ and the
property that
$g(\phi_\alpha\cdot,\cdot)=-g(\cdot,\phi_\alpha\cdot)$, one proves
that $g(\psi_1 X,\phi_2 X)$, $g(\phi_3 X,\psi_1 X)$, and $g(X,\psi_1
X)$ vanish. Hence the previous relation becomes
\begin{align}\label{intermediate2}
-g(R_{X \phi_3 X}\phi_1 X,\phi_2 X) & = g(R_{X \phi_3 X}X,\phi_3 X) + \frac{c^2}{4}  g(X,\psi_1^2 X)g(\psi_1 \phi_3 X,\phi_2 X) \nonumber\\
&=-H_{3}(X) + \frac{c^2}{4} g(X_{{\mathcal E}^{4l}},X_{{\mathcal E}^{4l}})^2
\end{align}
since $g(X,\psi_{1}^2 X)g(\psi_1 \phi_3 X,\phi_2
X)=-g(X,\phi_{1}^{2}X_{{\mathcal E}^{4l}})g(\phi_{2}X_{{\mathcal
E}^{4l}},\phi_{2}X)=g(X_{{\mathcal E}^{4l}},X)^2=g(X_{{\mathcal
E}^{4l}},X_{{\mathcal E}^{4l}})^2$. Making cyclic permutations of
$\{1,2,3\}$, one gets
\begin{gather}
-g(R_{X \phi_{1}X}\phi_{2}X,\phi_{3}X)=-H_{1}(X) + \frac{c^2}{4}g(X_{{\mathcal E}^{4l}},X_{{\mathcal E}^{4l}})^2 \label{intermediate3}\\
-g(R_{X \phi_{2}X}
\phi_{3}X,\phi_{1}X)=-H_{2}(X) + \frac{c^2}{4}g(X_{{\mathcal E}^{4l}},X_{{\mathcal E}^{4l}})^2. \label{intermediate4}
\end{gather}
Then by summing  \eqref{intermediate2}, \eqref{intermediate3},
\eqref{intermediate4}, the claim follows from the Bianchi identity.
\end{proof}

The notion of \emph{horizontal sectional curvature}
(\cite{konishi76}) plays in the context of $3$-structures the same
role played by the $\phi$-sectional curvature in  contact metric
geometry.  Let $(\phi_\alpha,\xi_\alpha,\eta_\alpha,g)$ be an almost
contact $3$-structure on $M$. Let $X$ be a horizontal vector at a
point $x$. Then one can consider the $4$-dimensional subspace
${\mathcal H}_x(X)$ of $T_{x}M$ defined by ${\mathcal
H}_x(X)=\left\langle X , \phi_1 X, \phi_2 X, \phi_3 X
\right\rangle$. ${\mathcal H}_x(X)$ is called the \emph{horizontal
section} determined by $X$. If the sectional curvature for any two
vectors belonging to  ${\mathcal H}_x(X)$ is a constant $k(X)$
depending only upon the fixed horizontal vector $X$ at $x$, then
$k(X)$ is said to be the horizontal sectional curvature with respect
to $X$ at $x$. Now let $X$ be an arbitrary horizontal vector field
on $M$. If the horizontal section ${\mathcal H}_x(X)$ at any point
$x$ of $M$ has a horizontal sectional curvature whose value $k(X)$
is independent of $X$, we say that the manifold $M$ is of
\emph{constant horizontal sectional curvature} at $x$. It is known
(\cite{konishi76}) that a 3-Sasakian manifold has constant
horizontal sectional curvature if and only if it has constant
curvature $1$. We now consider the 3-quasi-Sasakian setting.

\begin{theorem}
A 3-quasi-Sasakian manifold has constant horizontal sectional curvature
if and only if it is either 3-$c$-Sasakian or 3-cosymplectic. In the first case it
is a space of constant curvature $c^2/4$, in the latter it is flat.
\end{theorem}
\begin{proof}
We distinguish the case when $M$ is $3$-cosymplectic and $M$ is
$3$-quasi-Sasakian of rank $4l+3$. \ Let $M$ be a 3-cosymplectic
manifold of constant horizontal sectional curvature $k$ and let $x$
be a point of $M$. There exists a local Riemannian submersion $\pi$
defined on an open neighborhood of $x$ with base space a
hyper-K\"{a}hler manifold $(M',J'_\alpha,g')$. We recall the
well-known O'Neill formula (\cite{oneill66})  relating the sectional
curvatures of the total and  base spaces
\begin{equation}\label{oneillformula}
K(Y,Z)=K'(Y,Z) - 3\left\|A_{Y}Z\right\|=K'(Y,Z),
\end{equation}
$A$ denoting the O'Neill tensor, which in this case vanishes
identically since the distribution $\mathcal H$ is integrable.  As
the value of $k$ does not depend of the horizontal section
${\mathcal H}_{x}(X)$ at $x$, we can choose $X$ to be a basic vector
field. Since for any $\alpha,\beta\in\left\{1,2,3\right\}$,
${\mathcal L}_{\xi_\alpha}\phi_\beta=0$, ${\mathcal H}_{x}(X)$
projects to a horizontal section ${\mathcal H}_{x'}(X')$ on
$x'=\pi(x)$. Then, \eqref{oneillformula} implies that $M'$ has
constant horizontal sectional curvature $k$. It is well known that a
hyper-K\"{a}hler manifold of constant horizontal sectional curvature
is flat, hence by using \eqref{oneillformula} again we get that $M$
is horizontally flat. On the other hand, for any $Z\in\Gamma(TM)$,
we have $K(Z,\xi_\alpha)=0$ (cf. \cite[Lemma 2]{goldberg}). Thus $M$
is flat. Let us now suppose that $M$ is a 3-quasi-Sasakian manifold
of rank $4l+3$ with constant horizontal sectional curvature $k$. By
definition of horizontal sectional curvature,
$k=k(X)=H_{1}(X)=H_{2}(X)=H_{3}(X)$. Suppose the rank of $M$ is not
maximal, that is ${\mathcal E}^{4l}$ does not coincide with
$\mathcal H$. Then, from \eqref{sectional3qS}, we get that
$k(X)=\frac{c^2}{4}$ for $X\in\Gamma({\mathcal E}^{4l})$ and
$k(X)=0$ for $X\in\Gamma({\mathcal E}^{4m})$. This is in contrast
with the fact that the value of $k$ does not depend of $X$. Thus $M$
is necessarily of maximal rank and $k=\frac{c^2}{4}$. Hence, due to
\cite[Corollary 4.4]{ijm09}, $M$ is $3$-$c$-Sasakian. Observe now that one can apply a homothety to the
given structure, that is a change of the structure tensors of the
following type
\begin{equation}\label{deformation}
\bar\phi_\alpha:=\phi_\alpha, \quad \bar\xi_\alpha:=\frac{2}{c}\xi_\alpha, \quad
\bar\eta_\alpha:=\frac{c}2 \eta_\alpha, \quad \bar{g}:=\frac{c^2}{4} g,
\end{equation}
Then it is easy to check that the resulting structure $(\bar\phi_\alpha,\bar\xi_\alpha,\bar\eta_\alpha,\bar g)$ is $3$-Sasakian and its horizontal sectional curvature is proportional
to that of $(\phi_\alpha,\xi_\alpha,\eta_\alpha,g)$. Therefore, due
to \cite{konishi76},
$(M,\bar\phi_\alpha,\bar\xi_\alpha,\bar\eta_\alpha,\bar g)$ is a
space of constant sectional curvature and therefore the same is true
for $(M,\phi_\alpha,\xi_\alpha,\eta_\alpha,g)$. Its sectional
curvature is $k=\frac{c^2}{4}$.
\end{proof}

%\noindent{\bf Acknowledgments}
%
%\noindent Research partially supported by CMUC and FCT (Portugal), through
%European program COMPETE/FEDER, grants PTDC/MAT/099880/2008,
%MTM2009-13383 (A.D.N.), and SFRH/BPD/31788/2006 (I.Y.).

\end{document}